\documentclass{amsproc}

\usepackage[a4paper,top=31mm,bottom=31mm,left=30mm,right=30mm]{geometry}
\usepackage{amsmath,amssymb,amsthm,mathtools}
\usepackage{microtype}
\usepackage{enumitem,placeins}
\usepackage{hyperref}

\newtheorem{theorem}{Theorem}
\newtheorem{lemma}{Lemma}
\newtheorem{corollary}{Corollary}
\newtheorem{proposition}{Proposition}
\theoremstyle{remark}
\newtheorem{remark}{Remark}

\DeclareMathOperator{\JS}{JS}
\DeclareMathOperator{\Jc}{Jc}

\newcommand{\rz}{{\rm RZ}}

\newcommand{\lrf}[1]{\lfloor #1\rfloor}

\newcommand{\seps}{\prec}

\begin{document}

\title[Jacobi--Stirling numbers and Jacobi--Stirling descent polynomials]
{Binomial expansions of Jacobi–Stirling numbers and real-rootedness of Jacobi--Stirling descent polynomials}
\author{Shi-Mei Ma}
\address{School of Mathematics and Statistics, Shandong University of Technology, Zibo, Shandong 255000, P.R. China}
\email{shimeimapapers@163.com (S.-M. Ma)}
\author{Ming-Xin Wang}
\email{mingxinwang2026@163.com (M.-X. Wang)}
\address{School of Mathematics and Statistics, Shandong University of Technology, Zibo, Shandong 255000, P.R. China}



\date{}

\dedicatory{}

\begin{abstract}
In this paper, we first expand fixed diagonals of the Jacobi--Stirling numbers
in a binomial basis. For the second kind, the expansion coefficients
are polynomials in $z+1$ with nonnegative integer coefficients.
We give a recurrence and a signed-partition interpretation
for these coefficients. The same holds for the differences between corresponding
unsigned first-kind and second-kind coefficients. We then prove that every nonzero nonnegative linear combination
of the descent polynomials over Jacobi--Stirling permutations with a fixed number
of deleted barred letters has only simple negative zeros
when that number is one, two, or three. The same holds when exactly one or two barred letters are retained. 
Thus we verify five infinite families in a conjecture of
Gessel, Lin and Zeng. Finally, using insertion operators, we find that 
every descent polynomial over Jacobi--Stirling permutations with a fixed number of deleted barred letters is top heavy and has
an increasing left half.
\end{abstract}
\maketitle

%%%%%%%%%%%%%%%%%%%%%%%%%%%%%%%%%%%%%%%%%%%%
\section{Introduction}
%%%%%%%%%%%%%%%%%%%%%%%%%%%%%%%%%%%%%%%%%%%%
The unsigned first kind and second kind Jacobi--Stirling numbers are defined by
\begin{equation*}
\begin{aligned}
&\prod_{i=0}^{n-1}\bigl(x-i(i+z)\bigr)=\sum_{k=0}^{n}(-1)^{n+k}\mathrm{Jc}(n,k;z)x^k,\\
&x^n=\sum_{k=0}^{n}\mathrm{JS}(n,k;z)
  \prod_{i=0}^{k-1}\bigl(x-i(i+z)\bigr),
\end{aligned}
\end{equation*}
with the empty product equal to $1$,
see~\cite[p.~21, (3.4)]{LinZeng} for instance.
Equivalently, they satisfy the recurrence relations
\begin{equation}\label{recu}
\begin{aligned}
&\mathrm{Jc}(n,k;z)
=\mathrm{Jc}(n-1,k-1;z)
  +(n-1)(n-1+z)\mathrm{Jc}(n-1,k;z),\\
&\mathrm{JS}(n,k;z)
=\mathrm{JS}(n-1,k-1;z)
  +k(k+z)\mathrm{JS}(n-1,k;z),
\end{aligned}
\end{equation}
for $n,k\ge1$, with the initial conditions
$\mathrm{JS}(0,0;z)=\mathrm{Jc}(0,0;z)=1$
and $\mathrm{JS}(n,0;z)=\mathrm{JS}(0,n;z)
=\mathrm{Jc}(n,0;z)=\mathrm{Jc}(0,n;z)=0$ for $n\geq 1$.
We set $\JS(n,k;z)=\Jc(n,k;z)=0$ when $k<0$ or $k>n$.
 When $z=1$, Jacobi-Stirling numbers reduce to 
Legendre-Stirling numbers, see~\cite{AndrewsLittlejohn,Mongelli,MongelliTP,Zhu2024}.

Gelineau and Zeng \cite[p.~6, Theorem~3]{GelineauZeng}
gave a signed-partition interpretation of the coefficients
of $\mathrm{JS}(n,k;z)$ in powers of $z+1$. In this paper, we prove the same interpretation for the coefficients in the binomial expansion of
$\mathrm{JS}(n,n-r;z)$, where $r$ is fixed.
For $r=1$, the recurrence~\eqref{recu} gives
\begin{equation}\label{eq:firstdiagonal}
 \mathrm{JS}(n,n-1;z)
   = \sum_{j=1}^{n-1}j(j+z)
   = (z+1)\binom{n}{2}+2\binom{n}{3},
\end{equation}
which appears in~\cite[Example~4.3]{AndrewsEtAl}.
For $r=2$, it is routine to check that
\begin{equation}\label{eq:secondExample}
\begin{aligned}
 \mathrm{JS}(n,n-2;z)
   ={}& (z+1)^2\binom{n}{3}+\bigl(3(z+1)^2+10(z+1)+4\bigr)\binom{n}{4}\\
      &+\bigl(20(z+1)+32\bigr)\binom{n}{5}+40\binom{n}{6}.
\end{aligned}
\end{equation}
In both cases the binomial indices range from $r+1$
to $3r$, and the coefficient of $\binom{n}{3r}$ is
independent of $z$. The $z=1$ case of~\eqref{eq:secondExample} appears in~\cite[Eq.~(19)]{Egge}. 
In~\cite[Theorem~7]{MaMaYeh}, the Legendre-Stirling number $\mathrm{JS}(n,n-r;1)$ is expanded in a shifted binomial basis.
It is therefore natural to seek binomial expansions of Jacobi-Stirling numbers.

For $n\ge1$, write $[n]=\{1,\ldots,n\}$ and put
$[\pm n]_0=\{0\}\cup\{\pm i:i\in[n]\}$.
Following~\cite[Definition~1]{GelineauZeng},
a \emph{signed $k$-partition} of $[\pm n]_0$ is a partition of $[\pm n]_0$
into $k+1$ nonempty blocks $B_0,B_1,\ldots,B_k$ satisfying
the following conditions:
\begin{enumerate}
\item The block $B_0$ contains $0$, but contains no pair
$i,-i$.
\item For each $j=1,\ldots,k$, the block $B_j$ contains both signs of its smallest absolute value and no other pair $i,-i$. 
\end{enumerate}
The block containing $0$ is called the \emph{zero block}, and the other blocks are unordered.
For example, $$\bigl\{
  \{0,3,-4\},\
  \{-1,1,4\},\
  \{-2,2,-3\}
\bigr\}$$
is a signed $2$-partition of $[\pm4]_0$.
Its zero block is $\{0,3,-4\}$.
The smallest absolute values in the two nonzero blocks
are $1$ and $2$, respectively, and each occurs with both signs.

Our first main result is the following.
\begin{theorem}\label{thm:main01}
For $n\geq r\geq 1$, there are polynomials $b_{r,i}(z)$ such that
\begin{equation}\label{eq:secondExpansion}
 \JS(n,n-r;z)=\sum_{i=r+1}^{3r}b_{r,i}(z)\binom{n}{i}.
\end{equation}
Each $b_{r,i}(z)$ is a polynomial in $z+1$ with nonnegative integer coefficients, and its
coefficient of $(z+1)^j$ counts signed $(i-r)$-partitions of $[\pm i]_0$
whose zero block consists of $0$ and exactly $j$ negative
entries, and every nonzero block has at least three elements.
For $r\ge1$ and $r+2\le i\le3r+3$, 
\begin{equation}\label{eq:brec}
\begin{aligned}
b_{r+1,i}(z)
={}&(i-r-1)(i-r-1+z)b_{r,i-1}(z)+(i-1)(2i-2r-3+z)b_{r,i-2}(z)\\
 &+(i-1)(i-2)b_{r,i-3}(z),
\end{aligned}
\end{equation}
with $b_{1,2}(z)=z+1,~b_{1,3}(z)=2$, and $b_{r,j}(z)=0$ when $j<r+1$ or $j>3r$.
\end{theorem}

In Section~\ref{sec:03}, we deduce the following expansion and show that
$c_{r,i}(z)-b_{r,i}(z)$ is a polynomial in $z+1$ with nonnegative integer coefficients:
\begin{equation*}
 \Jc(n,n-r;z)
 =\sum_{j=r+1}^{3r}c_{r,j}(z)\binom nj.
\end{equation*}

\begin{table}[htbp]
\centering
\caption{The descent polynomials $A_{k,i}(x)$ for $1\le k\le3$}
\label{tab:descent-small}
\begin{tabular}{c c l}
\hline
$k$ & $i$ & $A_{k,i}(x)$ \\
\hline
1 & 0 & $1+x$ \\
  & 1 & $1$ \\
2 & 0 & $1+14x+21x^2+4x^3$ \\
  & 1 & $2+12x+6x^2$ \\
  & 2 & $1+2x$ \\
3 & 0 & $1+75x+603x^2+1065x^3+460x^4+36x^5$ \\
  & 1 & $3+114x+501x^2+436x^3+66x^4$ \\
  & 2 & $3+55x+116x^2+36x^3$ \\
  & 3 & $1+8x+6x^2$ \\
\hline
\end{tabular}
\end{table}

Let $f_k(n;z)=\JS(n+k,n;z)$. For $0\leq i\leq k$, let $p_{k,i}(n)$ be the coefficient of $z^i$ in $f_k(n;z)$. Following
Gessel, Lin and Zeng~\cite[Theorem~1]{GesselLinZeng}, the {\it Jacobi-Stirling descent polynomial} $A_{k,i}(x)$ can be defined by
$$\sum_{n\geq 0}p_{k,i}(n)x^n=\frac{xA_{k,i}(x)}{(1-x)^{3k-i+1}},$$
where $k\geq 1$. Here we remove the common
factor $x$ from the numerators used in~\cite{GesselLinZeng}, see Table~\ref{tab:descent-small}.

Define the multiset $M_k=\{\overline{1},1,1,\overline{2},2,2,\ldots,\overline{k},k,k\}$,
where the elements are ordered by
$$\overline{1}<1<\overline{2}<2<\cdots <\overline{k}<k.$$
Let $[\overline{k}]=\{\overline{1},\overline{2},\ldots,\overline{k}\}$. For any subset $S\subseteq [\overline{k}]$, set $M_{k,S}=M_k\setminus S$.
A permutation of $M_{k,S}$ is a {\it Jacobi-Stirling permutation} if for each $i$, $1\leq i\leq k$, all entries between the two occurrences of the unbarred $i$ are larger than $i$. Let $\operatorname{JSP}_{k,S}$ be the set of Jacobi-Stirling permutations of $M_{k,S}$, and let
$$\operatorname{JSP}_{k,i}=\bigcup_{\substack{S\subseteq [\overline{k}]\\ |S|=i}}\operatorname{JSP}_{k,S},$$
where the union is taken over all choices
of $i$ deleted barred letters.
Let $w=w_1w_2\cdots w_m$ be a word on a totally ordered alphabet. We say that $w_i$ is a descent of $w$ if $w_i>w_{i+1}$, where $i\in[ m-1]$. Let $\operatorname{des}(w)$ be the number of descents of $w$. Let
$$A_{k,S}(x)=\sum_{w\in\operatorname{JSP}_{k,S}}x^{\operatorname{des}(w)}.$$

The reader is referred to~\cite{LinMaZhang} and~\cite{MMY} 
for some multivariate refinements of the Jacobi-Stirling descent polynomials.
Gessel, Lin and Zeng~\cite[Theorem~2]{GesselLinZeng} proved that $A_{k,i}(x)$ enumerates $\operatorname{JSP}_{k,i}$ by descents. Hence
$$A_{k,i}(x)=\sum_{\substack{S\subseteq[\overline{k}]\\|S|=i}}
                    A_{k,S}(x).$$
They observed that a theorem of Brenti gives the
real-rootedness for each $A_{k,S}(x)$ and conjectured the same property for $A_{k,i}$
\cite[Conjecture~15]{GesselLinZeng}.
In particular, both $A_{k,0}(x)$ and $A_{k,k}(x)$ are real-rooted.
For the other cases, real-rootedness of the summands does not by
itself imply real-rootedness of their sum.
Our second main result establishes five infinite families of their conjecture.
\begin{theorem}\label{thm:deletions}
For $1\le i\le k$ with $i\in\{1,2,3,k-2,k-1\}$, the polynomial $A_{k,i}(x)$ is real-rooted.
\end{theorem}

A nonnegative sequence $a_0,a_1,\ldots,a_n$ is {\it top heavy} if
$a_k\leq  a_{n-k}$
 for all $0\le k<{n/2}$. We call a polynomial top heavy if its coefficient sequence is
top heavy. If no equality sign occurs, we call the polynomial strictly top heavy.
Top-heavy inequalities occur naturally in algebraic and combinatorial geometry, where they are often consequences of Lefschetz-type injectivity~\cite{HW}.
Our final main result establishes analogous inequalities for the
coefficients of Jacobi--Stirling descent polynomials.
\begin{theorem}\label{thm:top-heavy01}
Let $k\ge2$ and $0\le i\le k$. The Jacobi-Stirling descent polynomial $A_{k,i}(x)$ is strictly top heavy and has an
increasing left half. More precisely, write $A_{k,i}(x)=\sum_{j=0}^{d}a_jx^j$, where $d=2k-i-1$.
Then $a_j<a_{d-j}$ for $0\le j<d/2$ and $a_0<a_1<\cdots<a_{\lfloor d/2\rfloor}$.
\end{theorem}

The polynomials $A_{k,k}(x)$ are the well-studied second-order Eulerian polynomials over the set of Stirling permutations of order $k$. They satisfy the 
recurrence relation
$$A_{k+1,k+1}(x)=(1+2kx)A_{k,k}(x)+x(1-x)A_{k,k}'(x),$$ 
with the initial conditions $A_{1,1}(x)=1$ and $A_{2,2}(x)=1+2x$, see~\cite{GesselStanley,MMY} for instance. 
A special case of Theorem~\ref{thm:top-heavy01} says that $A_{k,k}(x)$ are strictly top heavy and have
an increasing left half for $k\ge2$.
In Sections~\ref{sec:02},~\ref{sec:04} and~\ref{section5}, we shall prove Theorem~\ref{thm:main01},~\ref{thm:deletions} and~\ref{thm:top-heavy01}, respectively. 
%%%%%%%%%%%%%%%%%%%%%%%%%%%%%%%%%%%%%%%%%%%%%%%%%%%%%%%%%%%
\section{Proof of Theorem~\ref{thm:main01}}\label{sec:02}
%%%%%%%%%%%%%%%%%%%%%%%%%%%%%%%%%%%%%%%%%%%%%%%%%%%%%%%%%%
We begin with the weighted signed partition interpretation of $\JS(n,k;z)$.
\begin{lemma}[{\cite[p.~6, Theorem~3]{GelineauZeng}}]
\label{lem:weighted}
Consider signed $k$-partitions of $[\pm n]_0$
whose zero block contains no positive entry. Give such a partition weight $(z+1)^j$, where $j$
is the number of negative entries in its zero block.
Then their total weight is $\JS(n,k;z)$.
\end{lemma}
\begin{proof}
When $n=0$, the only partition is $\{\{0\}\}$.
If $k=0<n$, the zero block would contain both $1$ and $-1$, contrary to the definition.
If $k>n$, there is no partition.

Now let $1\le k\le n$, and assume the interpretation holds
for $n-1$.
If $n$ and $-n$ belong to the same block, that block must be
$\{n,-n\}$.
Deleting this block gives a signed $(k-1)$-partition of
$[\pm(n-1)]_0$ with the same weight.
Conversely, adjoining $\{n,-n\}$ reverses this operation. Thus this case contributes $\JS(n-1,k-1;z)$.
Otherwise, $n$ and $-n$ lie in different blocks.
Neither is the smallest absolute value in a nonzero block, so deleting
them leaves a signed $k$-partition of $[\pm(n-1)]_0$.
To reverse this deletion, place $n$ in any of the $k$
nonzero blocks. Then place $-n$ either in the zero block,
with the weight factor $z+1$, or in one of the other
$k-1$ nonzero blocks.
These insertions leave all block minima unchanged.
Their contribution is 
\[
\bigl(k(z+1)+k(k-1)\bigr)\JS(n-1,k;z)
  =k(k+z)\JS(n-1,k;z).
\]
The total weights satisfy the second recurrence
in~\eqref{recu}. This completes the induction.
\end{proof}

\begin{proof}[Proof of Theorem~\ref{thm:main01}]
(A) We first prove the expansion~\eqref{eq:secondExpansion}.
Start with a signed $(n-r)$-partition of $[\pm n]_0$
whose zero block contains no positive entry.

Delete all two-element nonzero blocks, each of the form $\{s,-s\}$. Suppose that $i-r$ nonzero blocks remain.
Then $n-i$ blocks have been deleted, each of the form
$\{s,-s\}$, so $i$ positive entries and their negatives remain.
Relabel the remaining positive entries, in increasing order,
by $1,\ldots,i$, and relabel their negatives accordingly.
This gives a signed $(i-r)$-partition of $[\pm i]_0$
whose nonzero blocks all have at least three elements.
The number of negative entries in the zero block is unchanged,
so the weight is preserved. If no nonzero block remained, then $i=r\ge1$,
and the zero block would contain both $1$ and $-1$,
contrary to the definition.
Hence $i\ge r+1$.
The $i-r$ nonzero blocks each contain at least three
elements. Since there are only $2i$ elements other than $0$,
$3(i-r)\le2i$ and hence $i\le3r$.
Conversely, take such a partition of $[\pm i]_0$ and
choose an $i$-element subset of $[n]$.
Relabel the positive entries increasingly by the chosen
elements, and relabel their negatives accordingly.
For each unused $s\in[n]$, add the block $\{s,-s\}$.
The added blocks are exactly the two-element nonzero blocks,
so this reverses the preceding construction.
Thus each partition of $[\pm i]_0$ gives $\binom ni$
signed $(n-r)$-partitions of $[\pm n]_0$, all with the same weight.
Let $b_{r,i}(z)$ be the total weight of these signed
$(i-r)$-partitions of $[\pm i]_0$.
The preceding correspondence gives
\[
\JS(n,n-r;z)
  =\sum_{i=r+1}^{3r}b_{r,i}(z)\binom ni.
\]
The coefficient of $(z+1)^j$ in $b_{r,i}(z)$ therefore
counts such partitions whose zero block only consists of $0$ and $j$ negative
entries. Thus $b_{r,i}(z)$ is a polynomial
in $z+1$ with nonnegative integer coefficients. 

(B) We now prove the recurrence of $b_{r,i}(z)$.
For $r=1$ and $i=2$, the only partition is
\[
\bigl\{\{0,-2\},\{-1,1,2\}\bigr\},
\]
and so $b_{1,2}(z)=z+1$.
For $r=1$ and $i=3$, the only two partitions are
\[
\begin{aligned}
&\bigl\{\{0\},\{-1,1,3\},\{-2,2,-3\}\bigr\},~~\{\{0\},\{-1,1,-3\},\{-2,2,3\}\bigr\}.
\end{aligned}
\]
Both have weight $1$, so $b_{1,3}(z)=2$.
Consider a partition contributing to $b_{r+1,i}(z)$.
The entry $i$ lies in a nonzero block.
If this block also contained $-i$, then $i$ would be
its smallest absolute value. Since $i$ is also the largest
absolute value, the block would be $\{i,-i\}$,
contrary to the requirement of at least three elements.
Thus $-i$ lies in the zero block or in another nonzero block.
Delete $i$ and $-i$, and then delete the nonzero blocks
that now have two elements.
There are zero, one, or two such blocks, and let $h$ be their
number. 

\begin{itemize}
\item[$(c_1)$] Consider $h=0$.
Insert $i$ into one of the $i-r-1$ nonzero blocks.
Then insert $-i$ either into the zero block, with weight
factor $z+1$, or into one of the other $i-r-2$ nonzero
blocks. The contribution is $(i-r-1)[(z+1)+i-r-2]b_{r,i-1}(z)$.

\item[$(c_2)$] Consider $h=1$.
Take a partition contributing to $b_{r,i-2}(z)$
and choose $a\in[i-1]$.
Before adjoining $\{a,-a\}$, relabel the entries as follows:
increase each positive entry at least $a$ by $1$,
and decrease each negative entry at most $-a$ by $1$.
Leave all other entries unchanged.
For example, if $i=5$ and $a=2$, we relabel 
$\{1,2,3\}$ as $\{1,3,4\}$ and relabel
$\{-1,-2,-3\}$ as $\{-1,-3,-4\}$.
We can then adjoin $\{2,-2\}$ without repeating any entry.
After adjoining $\{a,-a\}$, we obtain a partition of
$[\pm(i-1)]_0$. Note that the deleted block was either $\{a,-a,i\}$ or
$\{a,-a,-i\}$.
In the first case, $-i$ may be inserted into the zero
block or any of the $i-r-2$ remaining nonzero blocks,
giving the factor $(z+1)+(i-r-2)$.
In the second case, $i$ may be inserted into any of
these $i-r-2$ nonzero blocks.
There are $i-1$ choices for $a$, so the contribution is
\[
(i-1)\bigl((z+1)+2(i-r-2)\bigr)b_{r,i-2}(z)
=(i-1)(2i-2r-3+z)b_{r,i-2}(z).
\]

\item[$(c_3)$] Consider $h=2$.
Take a partition contributing to $b_{r,i-3}(z)$.
Choose $1\le a<b\le i-1$ in $\binom{i-1}{2}$ ways.
Relabel the entries as in $(c_2)$, first with $a$
and then with $b$.
Adjoin the blocks $\{a,-a\}$ and $\{b,-b\}$.
Insert $i$ into one of these blocks and $-i$ into
the other, in either of two ways.
The weight is unchanged, so the contribution is
\[
2\binom{i-1}{2}b_{r,i-3}(z)
=(i-1)(i-2)b_{r,i-3}(z).
\]
\end{itemize}
Adding the three contributions gives the recurrence. This completes the proof.
\end{proof}

By~\eqref{eq:brec}, it is easy to verify that 
$b_{r,r+1}(z)=(z+1)^r$ and $b_{r,3r}(z)=\frac{(3r)!}{3^r r!}$.
For convenience, in Table~\ref{tab:coefficients} we list the coefficients
$b_{r,i}(z)$ for $1\le r\le3$.

\begin{table}[ht]
\centering
\caption{Coefficients of $(z+1)^j$ in $b_{r,i}(z)$ for $r\leq3$.}
\begin{tabular}{cc|rrrr}
\hline
$r$ & $i$ & $j=0$ & $j=1$ & $j=2$ & $j=3$\\
\hline
1 & 2 &    0 &    1 &   0 &  0\\
  & 3 &    2 &    0 &   0 &  0\\
\hline
2 & 3 &    0 &    0 &   1 &  0\\
  & 4 &    4 &   10 &   3 &  0\\
  & 5 &   32 &   20 &   0 &  0\\
  & 6 &   40 &    0 &   0 &  0\\
\hline
3 & 4 &    0 &    0 &   0 &  1\\
  & 5 &    8 &   28 &  34 & 10\\
  & 6 &  272 &  436 & 190 & 15\\
  & 7 & 1752 & 1372 & 210 &  0\\
  & 8 & 3584 & 1120 &   0 &  0\\
  & 9 & 2240 &    0 &   0 &  0\\
\hline
\end{tabular}\label{tab:coefficients}
\end{table}

\section{The difference between unsigned
first-kind and second-kind coefficients}\label{sec:03}
For $i\ge1$ and arbitrary $x$, set
$\binom{x}{i}=x(x-1)\cdots(x-i+1)/i!$, and set
$\binom{x}{0}=1$.
Gessel, Lin and Zeng~\cite[Proposition~8]{GesselLinZeng} proved a
reciprocity between the first-kind and second-kind diagonals.
Together with Theorem~\ref{thm:main01}, we can now give the following expansion.
\begin{corollary}\label{cor:firstExpansion}
For $r\ge1$ and $n\ge r$,
\begin{equation*}
\Jc(n,n-r;z)
 =(-1)^r\sum_{i=r+1}^{3r}
   b_{r,i}(-z)\binom{r-n}{i}.
\end{equation*}
\end{corollary}
\begin{proof}
Let $F_r(m)$ denote the polynomial that
agrees with $\mathrm{JS}(m,m-r;-z)$ for integers $m\geq r$, and put
$F_0(m)=1$.
By~\eqref{recu}, we find that
$F_r(m+1)-F_r(m)
 =(m-r+1)(m-r+1-z)F_{r-1}(m).$
Both sides are polynomials in $m$, so this identity also
holds at negative integers.
Put $G_r(n)=(-1)^rF_r(r-n)$.
Taking $m=r-n-1$ gives
\[
\begin{aligned}
G_r(n+1)-G_r(n)
 &=(-1)^{r+1}
   \bigl(F_r(r-n)-F_r(r-n-1)\bigr)=n(n+z)G_{r-1}(n).
\end{aligned}
\]
Moreover, $G_0(n)=1$ and
$G_r(0)=(-1)^r\JS(r,0;-z)=0$ for $r\ge1$.
Thus $G_r(n)$ satisfies the recurrence and initial conditions
for $\Jc(n,n-r;z)$, so they agree.
Substituting~\eqref{eq:secondExpansion} into $F_r(r-n)$
yields the desired formula.
\end{proof}

The negative-binomial identity gives
\[
\binom{r-n}{i}=(-1)^i\binom{n+i-r-1}{i}.
\]
Taking the coefficient of $t^i$ in
$(1+t)^n(1+t)^{i-r-1}$, Vandermonde's identity gives
\[(-1)^i\binom{n+i-r-1}{i}
=(-1)^i\sum_{j=r+1}^{i}
  \binom{i-r-1}{j-r-1}\binom{n}{j}.
\]
Substituting this identity into
Corollary~\ref{cor:firstExpansion}  gives
\begin{equation}\label{Jcnrz}
 \Jc(n,n-r;z)
 =\sum_{j=r+1}^{3r}c_{r,j}(z)\binom nj,
\end{equation}
where
\[
 c_{r,j}(z)
 =\sum_{i=j}^{3r}(-1)^{r+i}
   \binom{i-r-1}{j-r-1}b_{r,i}(-z).
\]

By~\eqref{recu}, we get $\Jc(n+1,n-r;z)-\Jc(n,n-r-1;z)
 =n(n+z)\Jc(n,n-r;z)$. Note that
$$
\Jc(n+1,n-r;z)
 =\Jc(n+1,(n+1)-(r+1);z)
   =\sum_i c_{r+1,i}(z)\binom{n+1}{i}.$$
Combining this with~\eqref{Jcnrz}, we get
\[
\begin{aligned}
&\Jc(n+1,n-r;z)-\Jc(n,n-r-1;z)\\
&=\sum_i c_{r+1,i}(z)\binom{n+1}{i}
  -\sum_i c_{r+1,i}(z)\binom ni\\
&=\sum_i c_{r+1,i}(z)
  \left(\binom{n+1}{i}-\binom ni\right)\\
&=\sum_i c_{r+1,i}(z)\binom n{i-1}.
\end{aligned}
\]
Therefore, we get $$\sum_i c_{r+1,i}(z)\binom n{i-1}=n(n+z)\Jc(n,n-r;z)=n(n+z)\sum_j c_{r,j}(z)\binom nj.$$
Note that
\[
(n+z)\binom nj
 =(j+z)\binom nj+(j+1)\binom n{j+1}.
\]
Multiplying by $n$ and applying the same binomial identity
to each term yields
\[
\begin{aligned}
n(n+z)\binom nj
=&(j+z)
 \left(j\binom nj+(j+1)\binom n{j+1}\right)+(j+1)
 \left((j+1)\binom n{j+1}
              +(j+2)\binom n{j+2}\right)\\
=&j(j+z)\binom nj
 +(j+1)(2j+1+z)\binom n{j+1}+(j+1)(j+2)\binom n{j+2}.
\end{aligned}
\]
We now compare coefficients of $\binom n{i-1}$.
The first term contributes when $j=i-1$, the second
when $j=i-2$, and the third when $j=i-3$.
Thus, for $r+2\le i\le3r+3$,
\begin{equation}\label{eq:brec02}
\begin{aligned}
c_{r+1,i}(z)&=(i-1)(i-1+z)c_{r,i-1}(z)+(i-1)(2i-3+z)c_{r,i-2}(z)\\
&+(i-1)(i-2)c_{r,i-3}(z),
\end{aligned}
\end{equation}
where we set $c_{r,j}(z)=0$ when $j<r+1$ or $j>3r$.
We can now conclude the following result.
\begin{theorem}\label{prop:comparison}
For $r\ge1$ and $r+1\le i\le3r$,
$c_{r,i}(z)-b_{r,i}(z)$ is a polynomial in $z+1$ with nonnegative integer coefficients.
\end{theorem}
\begin{proof}
We prove the comparison by induction on $r$. For $r=1$,
$c_{1,2}=b_{1,2}=z+1$ and $c_{1,3}=b_{1,3}=2$.
Suppose that every $c_{r,j}-b_{r,j}$ has nonnegative
integer coefficients in $z+1$.
In the recurrences~\eqref{eq:brec} and~\eqref{eq:brec02},
the third multipliers are equal. The differences between
the first and second multipliers are given as follows:
\[
\begin{aligned}
&(i-1)(i-1+z)
 -(i-r-1)(i-r-1+z)
   =r(2i-r-2+z),\\
&(i-1)(2i-3+z)
 -(i-1)(2i-2r-3+z)
   =2r(i-1).
\end{aligned}
\]
Hence
\[
\begin{aligned}
c_{r+1,i}-b_{r+1,i}
=&(i-1)(i-1+z)(c_{r,i-1}-b_{r,i-1})+(i-1)(2i-3+z)(c_{r,i-2}-b_{r,i-2})\\
 &+(i-1)(i-2)(c_{r,i-3}-b_{r,i-3})+r(2i-r-2+z)b_{r,i-1}
   +2r(i-1)b_{r,i-2}.
\end{aligned}
\]
When $i\ge r+2$, the identities
\[
\begin{aligned}
i-1+z&=(z+1)+(i-2),\\
2i-3+z&=(z+1)+(2i-4),\\
2i-r-2+z&=(z+1)+(2i-r-3)
\end{aligned}
\]
show that all multipliers have nonnegative
integer coefficients in $z+1$. Each $b_{r,j}$
has nonnegative coefficients in $z+1$ by
Theorem~\ref{thm:main01}, and the three differences have
this property by induction. Thus so does
$c_{r+1,i}-b_{r+1,i}$. Outside the stated range, both
coefficients are zero. This completes the proof.
\end{proof}

\section{Proof of Theorem~\ref{thm:deletions}}\label{sec:04}
%\hspace*{\parindent}
%%%%%%%%%%%%%%%%%%%%%%%%%%%%%%%%%%%%%%%%%%%
Fix a deletion set $S\subseteq [\overline{k}]$ and put $i=|S|$.
Recall that
$$A_{k,S}(x)=\sum_{\pi\in\operatorname{JSP}_{k,S}}
                   x^{\operatorname{des}(\pi)}.$$
We shall show that for $1\le i\le k$ with $i\in\{1,2,3,k-2,k-1\}$, 
every nonzero nonnegative linear combination of the
polynomials $A_{k,S}(x)$ has only simple
negative zeros. 

When $k=1$, the two permutations with no deleted elements are $\overline1\,11$ and $11\,\overline1$, and so $A_{1,\varnothing}=1+x$.
To obtain the permutations at $k=2$ with $\overline2$ deleted,
insert the pair $22$ into one of the four gaps of $\overline1\,11$ or $11\,\overline1$,
including the two ends. The first word contributes $1+3x$;
the second contributes $2x+2x^2$. Thus
$A_{2,\{\overline2\}}(x)=(1+3x)+(2x+2x^2)=1+5x+2x^2$.
 For $k\ge2$, fix
$S\subseteq[\overline{k-1}]$ and put $L=3(k-1)-|S|$.
Every permutation enumerated by $A_{k-1,S}(x)$ has the length $L$. Assume that it has $h$ descents.
Inserting $kk$ at a descent or at the end leaves the number
of descents unchanged. Each of the other $L-h$ insertions
increases it by one.
Its contribution is 
$(h+1)x^h+(L-h)x^{h+1}$, which yields that
\begin{equation}\label{AkS}
A_{k,S\cup\{\overline k\}}(x)
 =(1+Lx)A_{k-1,S}(x)+x(1-x)A'_{k-1,S}(x).
 \end{equation}
If $\overline k$ is retained, first insert $\overline k$ and then
insert $kk$. The first insertion has the same descent count as
above, and the second acts on words of length $L+1$.
In conclusion, we have 
\begin{equation}\label{eq:deletion-insertion}
 A_{k,S\cup\{\overline k\}}=D_LA_{k-1,S},~~
 A_{k,S}=D_{L+1}D_LA_{k-1,S}, 
\end{equation}
where $L=3(k-1)-|S|$ and $D_L$ is defined by 
\begin{equation}\label{DL}
D_Lf(x)=(1+Lx)f(x)+x(1-x)f'(x).
\end{equation}
Throughout the proof, the subscript in $D_L$ is the
actual length of the words before the insertion.
Different insertions may give the same operator $D_L$, which means only that they have the same effect on the
descent polynomial.

Let $\rz$ denote the set of nonzero real polynomials with
only real zeros. For an interval $I$, let $\rz(I)$ consist of those
polynomials in $\rz$ whose zeros all lie in $I$. For nonconstant $f,F\in\rz$, let
$r_1\ge r_2\ge\cdots$ and $s_1\ge s_2\ge\cdots$
be their respective zeros, counted with multiplicity.
We say that $f$ \emph{strictly interleaves} $F$, and
write $f\seps F$, if $\deg f\le\deg F\le\deg f+1$ and
$s_1>r_1>s_2>r_2>\cdots$.
For a nonzero constant $c$ and a linear polynomial $F$,
we set $c\seps F$.
Before showing Theorem~\ref{thm:deletions}, 
we provide two lemmas.
\begin{lemma}\label{lemma-zero}
Let $f$ be a polynomial of degree $d$ with positive leading
coefficient and simple negative zeros. For $L>d$, set
$D_Lf=(1+Lx)f+x(1-x)f'$.
Then $D_Lf$ has degree $d+1$, and $f\seps D_Lf$.
\end{lemma}
\begin{proof}
If $d=0$, then $D_Lf=f(1+Lx)$, so the claim follows from
$L>0$. Suppose $d\ge1$, and let $r_1>\cdots>r_d$ be the
zeros of $f$. If $a>0$ is the leading coefficient of $f$,
then the leading coefficient of $D_Lf$ is $(L-d)a>0$.
At each $r_j$, $D_Lf(r_j)=r_j(1-r_j)f'(r_j)$.
The factors $r_j(1-r_j)$ are negative, while $f'(r_j)$ has
sign $(-1)^{j-1}$. Thus $D_Lf(r_j)$ has sign $(-1)^j$.
Also $D_Lf(0)=f(0)>0$, and $D_Lf(x)$ has sign
$(-1)^{d+1}$ for sufficiently negative $x$.
Hence $D_Lf$ has a zero in each of
$(r_1,0),\quad (r_2,r_1),\quad\ldots,\quad
(r_d,r_{d-1}),\quad (-\infty,r_d)$.
These are $d+1$ distinct zeros, and so $f\seps D_Lf$.
\end{proof}

Every word in $\operatorname{JSP}_{k,S}$ has length $3k-|S|$.
The first occurrence of each unbarred letter is followed
by an equal or larger letter. Thus at least $k$ adjacent
pairs are not descents, giving at most $2k-|S|-1$ descents.
The nonincreasing word attains this bound, while the
nondecreasing word is the unique word with no descents.
Hence
\begin{equation}\label{degree}
\deg A_{k,S}(x)=2k-|S|-1,\qquad A_{k,S}(0)=1.
\end{equation}
Starting with $A_{1,\varnothing}(x)=1+x$ and
$A_{1,\{\overline1\}}(x)=1$, using~\eqref{AkS},~\eqref{eq:deletion-insertion}
and Lemma~\ref{lemma-zero}, we see that $A_{k,S}(x)\in\rz(-\infty,0)$ for $k\ge1$ and
$S\subseteq[\overline{k}]$, see also~\cite[p.~2000]{GesselLinZeng}.

\begin{lemma}\label{lem:insertion-comparison}
Let $f,g\in\rz(-\infty,0)$ have positive leading coefficients,
with $\deg g=\deg f+1$ and $f\seps g$.
If $L>\deg f$, then $D_Lf\seps D_{L+1}g$.
If $L>\deg g$, then $D_Lf\seps D_Lg$.
\end{lemma}
\begin{proof}
Put $d=\deg f$, and let $r_1>\cdots>r_{d+1}$ be the
zeros of $g$. We first prove $D_Lf\seps D_{L+1}g$
for $L>d$.
By Lemma~\ref{lemma-zero}, $D_{L+1}g$ has degree $d+2$
and $g\seps D_{L+1}g$. Moreover,
\[
 (D_{L+1}g)(r_1)=r_1(1-r_1)g'(r_1)<0,
 \qquad
 (D_{L+1}g)(0)=g(0)>0.
\]
Thus $D_{L+1}g$ has a zero in $(r_1,0)$.
Writing its zeros as $u_1>\cdots>u_{d+2}$, we obtain
$$0>u_1>r_1>u_2>\cdots>r_{d+1}>u_{d+2}.$$
We shall determine the sign of $D_Lf$ at these points.
Since $f\seps g$, both $f(r_i)$ and $g'(r_i)$ have
sign $(-1)^{i-1}$. Hence the partial fraction expansion is
\[
 \frac{f(x)}{g(x)}
 =\sum_{i=1}^{d+1}\frac{c_i}{x-r_i},
 \qquad
 c_i=\frac{f(r_i)}{g'(r_i)}>0.
\]
Differentiating gives
\begin{equation}\label{Dif}
 \left((1-x)\frac fg\right)'
 =-\sum_{i=1}^{d+1}
   \frac{c_i(1-r_i)}{(x-r_i)^2}<0
\end{equation}
for real $x$ away from the zeros of $g$.
By the product rule, $f'=\left(g\frac{f}{g}\right)'=g\left(\frac{f}{g}\right)'+\frac{f}{g}g'$.
Then we have 
\begin{align*}
D_{L}(f)&=(1+Lx)f+x(1-x)f'\\
&=(1+Lx)f+x(1-x)\left(g\left(\frac{f}{g}\right)'+\frac{f}{g}g'\right)\\
&=\frac{f}{g}\left((1+Lx)g+x(1-x)g'\right)+x(1-x)g\left(\frac{f}{g}\right)'\\
&=\frac{f}{g}D_L(g)+x(1-x)g\left(\frac{f}{g}\right)'.
\end{align*}

Since $D_{L}(g)=D_{L+1}(g)-xg$, we obtain
\begin{align*}
D_{L}(f)&=\frac{f}{g}D_{L+1}(g)-xf+x(1-x)g\left(\frac{f}{g}\right)'\\
&=\frac{f}{g}D_{L+1}(g)+xg\left(-\frac{f}{g}+(1-x)\left(\frac{f}{g}\right)'\right)\\
&=\frac fg\,D_{L+1}g
   +xg\left((1-x)\frac fg\right)'.
\end{align*}
Since $g(u_i)\ne0$, we may evaluate this identity at $u_i$:
\[
 (D_Lf)(u_i)
 =u_i g(u_i)\left((1-x)\frac fg\right)'\!(u_i).
\]
It follows from~\eqref{Dif} that $(D_Lf)(u_i)$ has the same sign as $g(u_i)$,
namely $(-1)^{i-1}$.
These signs alternate, so $D_Lf$ has a zero in each
$(u_{i+1},u_i)$, for $1\le i\le d+1$.
By Lemma~\ref{lemma-zero}, its degree is $d+1$.
The zeros just found are therefore all its zeros, and
each is simple. This proves $D_Lf\seps D_{L+1}g$.

Now suppose $L>\deg g$. By Lemma~\ref{lemma-zero},
$g\seps D_Lg$. Thus every zero $u$ of $D_Lg$ is
negative, and $g(u)\ne0$. Since $f=(f/g)g$, the product
rule gives
\[
\begin{aligned}
D_Lf
&=\frac fg\,D_Lg
  +x(1-x)g\left(\frac fg\right)'.
\end{aligned}
\]
At $x=u$, the term containing $D_Lg$ vanishes. Hence
\[
(D_Lf)(u)
 =u(1-u)g(u)\left(\frac fg\right)'(u).
\]
Differentiating the partial fraction expansion gives
\[
\left(\frac fg\right)'(x)
 =-\sum_{i=1}^{d+1}\frac{c_i}{(x-r_i)^2}<0
 \qquad (g(x)\ne0).
\]
Since $u(1-u)<0$, the values $(D_Lf)(u)$ and $g(u)$
have the same nonzero sign.
There is exactly one simple zero of $g$ between consecutive
zeros of $D_Lg$. Thus $D_Lf$ takes opposite signs at the
endpoints of each of the $d+1$ intervening intervals.
The intermediate value theorem gives a zero in each interval.
Since $\deg D_Lf=d+1$, these are all its zeros, and each
is simple and negative. Hence $D_Lf\seps D_Lg$.
The argument also applies when $f$ is constant.
\end{proof}

For $j\ge2$, let $H_j(x)$ be the descent polynomial of
the intermediate words obtained from
$\operatorname{JSP}_{j-1,\varnothing}$ by inserting
$\overline j$, but before inserting the pair $jj$.
We shall use $H_j$ only for this intermediate family.
Thus its underlying words always have length $3j-2$. Hence
\begin{equation}\label{eq:Hj}
 H_j=D_{3j-3}A_{j-1,0},
 \qquad
 A_{j,0}=D_{3j-2}H_j.
\end{equation}

For clarity, we divide the proof of Theorem~\ref{thm:deletions} into five subsections.
\subsection{The real-rootedness of $A_{k,1}(x)$}
\hspace*{\parindent}

For $1\le j\le k$, we shall prove that 
\begin{equation}\label{eq:one-deletion}
 A_{k,\{\overline j\}}\seps A_{k,0}.
\end{equation}

Fix $j\in[k]$. We first establish the comparison when the
letter $\overline j$ is introduced.
If $j=1$, then $A_{1,\{\overline1\}}=1$ and $A_{1,0}=1+x$, and
so the required relation holds.
When $j\ge2$, both constructions start with a full
Jacobi--Stirling permutation of order $j-1$, of length
$3j-3$. If $\overline j$ is deleted, we insert only the
pair $jj$, and hence $A_{j,\{\overline j\}}
   =D_{3j-3}A_{j-1,0}$.
If $\overline j$ is retained, inserting it gives
$H_j=D_{3j-3}A_{j-1,0}$,
and the subsequent insertion of $jj$ gives
$A_{j,0}=D_{3j-2}H_j$.
Thus $A_{j,\{\overline j\}}=H_j$
as polynomials. We do not identify the underlying word
families: the first consists of words of length $3j-1$,
whereas the intermediate words defining $H_j$ have length
$3j-2$. By Lemma~\ref{lemma-zero}, we get 
$H_j\seps A_{j,0}$.
Hence $A_{j,\{\overline j\}}\seps A_{j,0}$.
We now carry this comparison from order $j$ to order $k$.
Suppose $j\le s<k$ and $A_{s,\{\overline j\}}\seps A_{s,0}$.
The corresponding words have lengths $3s-1$ and $3s$.
At order $s+1$, both constructions retain
$\overline{s+1}$.
Inserting $\overline{s+1}$ gives $D_{3s-1}A_{s,\{\overline j\}}$ and
$D_{3s}A_{s,0}$.
By Lemma~\ref{lem:insertion-comparison},
\[
 D_{3s-1}A_{s,\{\overline j\}}
 \seps
 D_{3s}A_{s,0}.
\]
The new word lengths are $3s$ and $3s+1$.
Inserting the pair $(s+1)(s+1)$ and applying the same
lemma once more gives $$D_{3s}D_{3s-1}A_{s,\{\overline j\}}
 \seps
 D_{3s+1}D_{3s}A_{s,0},$$
that is, $A_{s+1,\{\overline j\}}\seps A_{s+1,0}$.
Repeating this step for $s=j,j+1,\ldots,k-1$ proves
\eqref{eq:one-deletion}.

For $k=1$, we have $A_{1,1}=1$. Assume $k\ge2$, and let
$\alpha_1<\alpha_2<\cdots<\alpha_{2k-1}$
be the zeros of $A_{k,0}$. Since the leading coefficient of $A_{k,\{\overline j\}}$ is positive,
$\operatorname{sgn}
 A_{k,\{\overline j\}}(\alpha_r)
 =(-1)^{2k-1-r}$.
The polynomial
$A_{k,\{\overline j\}}$ has one zero in each interval
$(\alpha_r,\alpha_{r+1})$. Therefore, the polynomial
$A_{k,1}(x)
 =\sum_{j=1}^kA_{k,\{\overline j\}}(x)$
has opposite signs at consecutive $\alpha_r$, and
it has a zero in each of the $2k-2$ intervals
$(\alpha_r,\alpha_{r+1})$, where $1\le r\le2k-2$.
Since $\deg A_{k,1}=2k-2$, these are all its zeros.
Thus $A_{k,1}(x)$ has only simple negative zeros.

\subsection{The real-rootedness of $A_{k,2}(x)$}
\hspace*{\parindent}

From~\eqref{eq:Hj}, we see that $H_k=D_{3k-3}A_{k-1,0}$.
We prove by induction on \(k\) that 
\begin{equation}\label{eq:two-deletion}
 A_{k,S}\seps H_k,
\end{equation}
where $S\subseteq[\overline k]$ and $|S|=2$.
For $k=2$, the only deletion set is
$S=\{\overline1,\overline2\}$.
The words counted by $A_{1,\{\overline1\}}=1$ have length
$2$. By~\eqref{eq:one-deletion}, we see that
$$A_{1,\{\overline1\}}=1
 \seps 1+x=A_{1,0}.$$
The first assertion of
Lemma~\ref{lem:insertion-comparison}, with $L=2$, gives
$$A_{2,S}=D_2(1)=1+2x
 \seps D_3(1+x)=1+5x+2x^2=H_2.$$

Let $k\ge3$, and assume~\eqref{eq:two-deletion} at order
$k-1$. Write $S=\{\overline a,\overline b\}$, where $a<b$.
First suppose that $b=k$. Put $T=\{\overline a\}$.
The words counted by $A_{k-1,T}$ have length $3k-4$.
Since $\overline k$ is deleted, inserting the pair $kk$
gives $A_{k,S}=D_{3k-4}A_{k-1,T}$.
The full words of order $k-1$ have length $3k-3$.
Inserting $\overline k$ gives $H_k=D_{3k-3}A_{k-1,0}$.
It follows from~\eqref{eq:one-deletion} that $A_{k-1,T}\seps A_{k-1,0}$.
Lemma~\ref{lem:insertion-comparison} therefore gives~\eqref{eq:two-deletion}.

Now suppose that $b<k$. Then
$S\subseteq[\overline{k-1}]$, and the induction hypothesis
gives $$A_{k-1,S}\seps H_{k-1}.$$
Both underlying word families have length $3k-5$.
On the left side, insert $\overline k$ and then $kk$:
\[
 A_{k-1,S}
 \xrightarrow[\overline k]{D_{3k-5}}
 D_{3k-5}A_{k-1,S}
 \xrightarrow[kk]{D_{3k-4}}
 A_{k,S}.
\]
On the right side, $H_{k-1}$ is the intermediate family
in the full construction of order $k-1$. Insert
$(k-1)(k-1)$ to complete that construction, and then
insert $\overline k$:
\[
 H_{k-1}
 \xrightarrow[(k-1)(k-1)]{D_{3k-5}}
 A_{k-1,0}
 \xrightarrow[\overline k]{D_{3k-3}}
 H_k.
\]
The two insertions are different, but both start with
words of length $3k-5$. The subscript of $D_L$ records
this length, not the letter or pair being inserted.
Hence the second assertion of
Lemma~\ref{lem:insertion-comparison} gives
$D_{3k-5}A_{k-1,S}\seps A_{k-1,0}$.
The two resulting word lengths are $3k-4$ and $3k-3$.
Applying the first assertion of the lemma to the next two
insertions gives $A_{k,S}\seps H_k$.
This proves~\eqref{eq:two-deletion}.

Let $\beta_1<\beta_2<\cdots<\beta_{2k-2}$
be the zeros of $H_k$.
For every $S$ with $|S|=2$, the polynomial $A_{k,S}$
has one zero in each interval
$(\beta_r,\beta_{r+1})$. Hence it has exactly
$2k-2-r$ zeros to the right of $\beta_r$, and
$\operatorname{sgn}A_{k,S}(\beta_r)
   =(-1)^{2k-2-r}$, which is
independent of $S$.
Therefore, the polynomial $A_{k,2}(x)=\sum_{|S|=2}A_{k,S}(x)$
has opposite signs at consecutive $\beta_r$.
It has a zero in each of the $2k-3$ intervals
$(\beta_r,\beta_{r+1})$, where $1\le r\le2k-3$.
Since $\deg A_{k,2}=2k-3$, these are all its zeros.
Thus $A_{k,2}(x)$ has only simple negative zeros.

\subsection{The real-rootedness of $A_{k,3}(x)$}
\hspace*{\parindent}

Let $S\subseteq[\overline k]$ and $|S|=3$.
We prove by induction on $k$ that
\begin{equation}\label{eq:three-deletion}
 A_{k,S}\seps A_{k-1,0}.
\end{equation}

For $k=3$, the only deletion set is $S=\{\overline1,\overline2,\overline3\}$.
The words counted by
$A_{2,\{\overline1,\overline2\}}$ have length $4$.
Since $\overline3$ is deleted, inserting $33$ gives
$A_{3,S}=D_4A_{2,\{\overline1,\overline2\}}$.
The intermediate words defining $H_2$ also have length
$4$. Inserting $22$ gives $A_{2,0}=D_4H_2$.
From~\eqref{eq:two-deletion}, we see that $A_{2,\{\overline1,\overline2\}}\seps H_2$.
The second assertion of
Lemma~\ref{lem:insertion-comparison} therefore gives
$A_{3,S}\seps A_{2,0}$.

Let $k\ge4$, and assume~\eqref{eq:three-deletion} at
order $k-1$. Write $S=\{\overline a,\overline b,\overline c\}$, where $a<b<c$.
First suppose that $c=k$. Put $T=\{\overline a,\overline b\}$.
The words counted by $A_{k-1,T}$ have length $3k-5$.
Since $\overline k$ is deleted, inserting $kk$ gives
$A_{k,S}=D_{3k-5}A_{k-1,T}$.
The intermediate words defining $H_{k-1}$ also have
length $3k-5$. Inserting $(k-1)(k-1)$ gives $A_{k-1,0}=D_{3k-5}H_{k-1}$.
From~\eqref{eq:two-deletion}, we see that $A_{k-1,T}\seps H_{k-1}$. By
the second assertion of
Lemma~\ref{lem:insertion-comparison}, we arrive at $A_{k,S}\seps A_{k-1,0}$, as desired.

Now suppose that $c<k$. Then
$S\subseteq[\overline{k-1}]$, and the induction hypothesis
gives $A_{k-1,S}\seps A_{k-2,0}$.
Both underlying word families have length $3k-6$.
On the left side, insert $\overline k$ and then $kk$:
\[
 A_{k-1,S}
 \xrightarrow[\overline k]{D_{3k-6}}
 D_{3k-6}A_{k-1,S}
 \xrightarrow[kk]{D_{3k-5}}
 A_{k,S}.
\]
On the right side, insert $\overline{k-1}$ and then
$(k-1)(k-1)$:
\[
 A_{k-2,0}
 \xrightarrow[\overline{k-1}]{D_{3k-6}}
 H_{k-1}
 \xrightarrow[(k-1)(k-1)]{D_{3k-5}}
 A_{k-1,0}.
\]
At the first step both operators act on words of length
$3k-6$. By the second assertion of
Lemma~\ref{lem:insertion-comparison}, we obtain
$D_{3k-6}A_{k-1,S}\seps H_{k-1}$.
Both resulting families have length $3k-5$.
A second application of the same assertion gives
$A_{k,S}\seps A_{k-1,0}$.
This proves~\eqref{eq:three-deletion}.

Let $\gamma_1<\gamma_2<\cdots<\gamma_{2k-3}$
be the zeros of $A_{k-1,0}$.
For every $S$ with $|S|=3$, the polynomial $A_{k,S}$
has one zero in each interval
$(\gamma_r,\gamma_{r+1})$. Hence it has exactly
$2k-3-r$ zeros to the right of $\gamma_r$, and
$\operatorname{sgn}A_{k,S}(\gamma_r)
   =(-1)^{2k-3-r}$, which is
independent of $S$.
Therefore, the polynomial $A_{k,3}(x)=\sum_{|S|=3}A_{k,S}(x)$
has opposite signs at consecutive $\gamma_r$.
It has a zero in each of the $2k-4$ intervals
$(\gamma_r,\gamma_{r+1})$, where $1\le r\le2k-4$.
Since $\deg A_{k,3}=2k-4$, these are all its zeros.
Thus $A_{k,3}(x)$ has only simple negative zeros.

\subsection{Retaining one barred letter and the real-rootedness of $A_{k,k-1}(x)$}
\hspace*{\parindent}

We prove by induction on $k$ that 
\begin{equation}\label{AKK}
A_{k,k}(x)\seps A_{k,S}(x),
\end{equation} where $|S|=k-1$.
For $k=1$, this is $1\seps 1+x$.
Let $k\ge2$, and assume that the assertion holds for $k-1$.
The words counted by $A_{k-1,k-1}$ have length $2k-2$.
Inserting $kk$ gives $A_{k,k}=D_{2k-2}A_{k-1,k-1}$. 

Fix $S$ with $|S|=k-1$. If $\overline k\notin S$, then all barred letters
among $[\overline{k-1}]$ are deleted, and hence
$S=[\overline{k-1}]$. To construct the words
counted by $A_{k,S}$, insert $\overline k$ and then
$kk$. Then $A_{k,S}
 =D_{2k-1}D_{2k-2}A_{k-1,k-1}$. Since
$\deg A_{k,k}=k-1<2k-1$, Lemma~\ref{lemma-zero}
gives $A_{k,k}\seps A_{k,S}$.

If $\overline k\in S$, put
$T=S\setminus\{\overline k\}$. Then $|T|=k-2$,
and the words counted by $A_{k-1,T}$ have length
$2k-1$. Inserting $kk$ gives $A_{k,S}=D_{2k-1}A_{k-1,T}$.
By induction, $A_{k-1,k-1}\seps A_{k-1,T}$.
Lemma~\ref{lem:insertion-comparison}, with
$L=2k-2$, now gives $D_{2k-2}A_{k-1,k-1}
 \seps D_{2k-1}A_{k-1,T}$,
which is $A_{k,k}\seps A_{k,S}$.

Finally, let $r_1>\cdots>r_{k-1}$ be the zeros of
$A_{k,k}$. At $r_j$, every $A_{k,S}$ with
$|S|=k-1$ has sign $(-1)^j$. Any nonzero
nonnegative linear combination has the same signs.
It is positive at $0$ and has sign $(-1)^k$ for
sufficiently negative $x$. It has one zero
in each of the intervals $(r_1,0),~(r_2,r_1),\ldots,~
(-\infty,r_{k-1})$.
Its degree is $k$, so these are all its zeros and
they are simple. In particular,
$A_{k,k}\seps A_{k,k-1}$.

\subsection{Retaining two barred letters and the real-rootedness of $A_{k,k-2}(x)$}
\hspace*{\parindent}

For $k=2$, we have $S=\varnothing$. Starting from
$A_{1,1}=1\seps A_{1,\varnothing}=1+x$, the two parts
of Lemma~\ref{lem:insertion-comparison} give
$A_{2,2}=D_2A_{1,1}
 \seps D_3A_{1,\varnothing}$,
and then $A_{3,3}=D_4A_{2,2}
 \seps D_4D_3A_{1,\varnothing}
 =A_{2,\varnothing}$.
This is the base case.
We now prove that $A_{k+1,k+1}\seps A_{k,S}$ by induction on $k\ge2$, where $|S|=k-2$.
Since the words counted by $A_{k,k}$ have length $2k$,
$A_{k+1,k+1}=D_{2k}A_{k,k}$.

If $\overline k\in S$, then $k\ge3$. We put
$T=S\setminus\{\overline k\}$, and so $|T|=k-3$.
The words counted by $A_{k-1,T}$ have length $2k$,
so $A_{k,S}=D_{2k}A_{k-1,T}$.
By induction, $A_{k,k}\seps A_{k-1,T}$.
Lemma~\ref{lem:insertion-comparison}, applied with
$L=2k$, gives $A_{k+1,k+1}\seps A_{k,S}$.

If $\overline k\notin S$, put $T=S$, so $|T|=k-2$.
The words counted by $A_{k-1,T}$ have length $2k-1$.
We insert $\overline k$ and then insert $kk$. Hence
$A_{k,S}=D_{2k}D_{2k-1}A_{k-1,T}$. From~\eqref{AKK}, we see that 
$A_{k-1,k-1}\seps A_{k-1,T}$.
Applying Lemma~\ref{lem:insertion-comparison} first
with $L=2k-2$ gives $A_{k,k}
 =D_{2k-2}A_{k-1,k-1}
 \seps D_{2k-1}A_{k-1,T}$.
Applying its second assertion with $L=2k$ now gives
$A_{k+1,k+1}\seps A_{k,S}$. The degree condition is satisfied since
$\deg(D_{2k-1}A_{k-1,T})=k<2k$.

Let $r_1>\cdots>r_k$ be the zeros of
$A_{k+1,k+1}$. Since
$A_{k+1,k+1}\seps A_{k,S}$ and $A_{k,S}$ has
positive leading coefficient, the value $A_{k,S}(r_j)$ has
sign $(-1)^j$, which is independent of $S$. Every nonzero
nonnegative linear combination has the same sign
at $r_j$. Such a sum is positive at $0$ and has
sign $(-1)^{k+1}$ for sufficiently negative $x$.
It has a zero in each of the intervals $(r_1,0),~ (r_2,r_1),\ldots,
(-\infty,r_k)$.
These $k+1$ distinct negative zeros account for
its degree, so they are all simple. In particular,
$A_{k,k-2}$ has only simple negative zeros.

\begin{corollary}\label{cor:weighted-deletions}
For $1\le i\le k$ with $i\in\{1,2,3,k-2,k-1\}$, let
\[
 F_{k,i}(x)=
 \sum_{\substack{S\subseteq[\overline k]\\ |S|=i}}
 c_SA_{k,S}(x),
\]
where $c_S$ are nonnegative real numbers, not all zero.
Then $F_{k,i}(x)$ has only simple negative zeros.
Moreover, we have
$$F_{k,1}\seps A_{k,0},~F_{k,2}\seps H_k,~F_{k,3}\seps A_{k-1,0},~A_{k,k}\seps F_{k,k-1},~A_{k+1,k+1}\seps F_{k,k-2}.$$
\end{corollary}

\begin{proof}
In each of the five cases above, all summands have the
same nonzero sign at every zero of the corresponding
comparison polynomial. In the two retained-letter cases, the signs at $0$ and
at sufficiently negative $x$ also give a zero in each
of the two outer intervals.
The same is therefore true of any
nonzero nonnegative linear combination. The intermediate
value theorem gives one zero in every intervening interval.
Since all summands have the same degree and positive
leading coefficient, these zeros account for the degree.
\end{proof}

\begin{remark}
The next case of real-rootedness may be reduced to a single interlacing problem. For $k\geq 5$, 
it is routine to check that
$A_{k,4}
 =
 D_{3k-6}
 \left(
 D_{3k-7}A_{k-1,4}+A_{k-1,3}
 \right)$.
By~\eqref{eq:three-deletion}, $A_{k-1,3}\seps A_{k-2,0}$.
Thus a sufficient condition for the real-rootedness of $A_{k,4}$ is
$D_{3k-7}A_{k-1,4}\seps A_{k-2,0}$.
We do not establish this additional comparison here.
\end{remark}

Recall that $A_{k,k}(x)$ are the second-order Eulerian polynomials.
We conclude this section with the following result.
\begin{proposition}
For $k\geq 1$, we have $$A_{k,k-1}(x)=kA_{k+1,k+1}(x)-\frac{x}{2}A'_{k+1,k+1}(x).$$
\end{proposition}
\begin{proof}
For $k=1$, the identity follows from
$A_{1,0}=1+x$ and $A_{2,2}=1+2x$.
Let \(k\ge2\), and assume the identity holds at order \(k-1\).
Every word counted by $A_{k-1,k-1}$ has length $2k-2$.
Inserting the pair $kk$ gives precisely the words counted
by $A_{k,k}$. Hence $A_{k,k}=D_{2k-2}A_{k-1,k-1}$.
A word counted by $A_{k,k-1}$ has exactly one barred
letter. 
We distinguish whether this letter is $\overline{k}$, and 
the insertion operators give the polynomial
identities
$$A_{k,k-1}=D_{2k-1}A_{k-1,k-2}
  +D_{2k-1}D_{2k-2}A_{k-1,k-1}=D_{2k-1}\bigl(A_{k-1,k-2}+A_{k,k}\bigr).$$

By the induction hypothesis, we have
$2\bigl(A_{k-1,k-2}+A_{k,k}\bigr)
 =2kA_{k,k}-xA'_{k,k}$.
For every polynomial $f$, note that
$D_{2k-1}(2kf-xf')
 =2kD_{2k}f-x(D_{2k}f)'$.
By linearity, it suffices to check this for $f=x^r$.
In that case, both sides equal
$(2k-r)\bigl((r+1)x^r+(2k-r-1)x^{r+1}\bigr)$.
Applying this identity with $f=A_{k,k}$, we obtain
\[
\begin{aligned}
 2A_{k,k-1}
 &=D_{2k-1}\bigl(2kA_{k,k}-xA'_{k,k}\bigr)\\
 &=2kD_{2k}A_{k,k}-x(D_{2k}A_{k,k})'\\
 &=2kA_{k+1,k+1}-xA'_{k+1,k+1},
\end{aligned}
\]
as desired. This completes the induction.
\end{proof}

\section{Top-heaviness of the Jacobi-Stirling descent polynomials}\label{section5}

We first record a basic consequence of the insertion
operator $D_L$ defined by~\eqref{DL}.
\begin{lemma}\label{lem:top-heavy-operator}
Let $f(x)$ be a polynomial of degree $d\ge1$ with positive coefficients, and let $L\ge d+1$ be an integer.
If $f(x)$ is strictly top heavy and has an increasing left half, then $D_Lf$ has the same two properties.
\end{lemma}

\begin{proof}
Let $f(x)=\sum_{i=0}^{d}a_ix^i$, and let $D_Lf(x)=\sum_{i=0}^{d+1}b_ix^i$.
Then 
\begin{equation}\label{eq:left-increasing}
 a_0< a_1<\cdots< a_{\lfloor d/2\rfloor},
\end{equation}
\begin{equation}\label{eq:top-heavy}
 a_i< a_{d-i}
 \qquad (0\le i<{d/2}).
\end{equation}
Using
$D_Lf=(1+Lx)f+x(1-x)f'$,
we obtain
\begin{equation}\label{eq:DL-coefficients}
 b_i=(i+1)a_i+(L-i+1)a_{i-1},
\end{equation}
where we set $a_{-1}=a_{d+1}=0$. 
Note that
\begin{equation*}
b_{i+1}-b_i=(i+2)(a_{i+1}-a_i)
   +(L-i+1)(a_i-a_{i-1}).
\end{equation*}
For $i< \lrf{d/2}$, by~\eqref{eq:left-increasing}, both terms are positive. 
There is one more comparison at the middle. If $d=2m+1$,
then $a_m<a_{m+1}$ by~\eqref{eq:top-heavy}, while
$a_{m-1}< a_m$ by~\eqref{eq:left-increasing}.
Thus $b_m< b_{m+1}$. If $d=2m$, rearranging the same
formula gives
\[
b_{m+1}-b_m
 =(m+2)(a_{m+1}-a_{m-1})
 +(L-2m-1)(a_m-a_{m-1})>0.
\]
Here $a_{m-1}< a_{m+1}$ follows from top-heaviness,
and $a_{m-1}< a_m$ follows from the increasing
left half.
Thus $D_Lf$ has an increasing left half.

We next prove top-heaviness. For $i=0$, we have $b_{d+1}-b_0=(L-d)a_d-a_0>0$,
since $L-d\ge1$ and $a_0< a_d$.
Now let $1\le i<{(d+1)/2}$. Using~\eqref{eq:DL-coefficients}, we get
\begin{equation*}
\begin{aligned}
 b_{d+1-i}-b_i&=(d+2-i)(a_{d+1-i}-a_{i-1})+(L-d+i)(a_{d-i}-a_i)\\
 &+(L-d-1)(a_i-a_{i-1}).
\end{aligned}
\end{equation*}
By the top-heaviness of $f(x)$, we have $a_{i-1}<a_{d-(i-1)}$, since $i-1<{d/2}$.
Thus the first difference is positive.
For the second difference, if $i<{d/2}$, then
$a_i< a_{d-i}$. The only remaining possibility is
that $d$ is even and $i=d/2$. In this case, $a_{d-i}-a_i=0$.
Finally, we have $a_i-a_{i-1}>0$, since
$i\le\lfloor d/2\rfloor$ and the left half is increasing.
Note that $d+2-i>0$, $L-d+i>0$ and $L-d-1\geq 0$. Therefore, $b_i< b_{d+1-i}$.
This completes the proof.
\end{proof}

\begin{theorem}\label{thm:top-heavy}
Let $k\ge2$ and $S\subseteq[\overline k]$. Set $d=2k-|S|-1$. Write
$$A_{k,S}(x)=\sum_{r=0}^{d}a_r(S)x^r.$$
Then $A_{k,S}(x)$ is strictly top heavy and has an increasing
left half. Consequently, 
the polynomial $A_{k,i}(x)$ has the same two properties for every $0\le i\le k$.
\end{theorem}
\begin{proof}
We prove the strict inequalities for each $A_{k,S}$.
At $k=2$, the four polynomials are
\[
\begin{aligned}
A_{2,\varnothing}&=1+14x+21x^2+4x^3, &
A_{2,\{\overline1\}}&=1+7x+4x^2,\\
A_{2,\{\overline2\}}&=1+5x+2x^2, &
A_{2,\{\overline1,\overline2\}}&=1+2x.
\end{aligned}
\]
Each polynomial is strictly
top heavy and has an increasing left half. Now let $k\ge3$, and assume these inequalities hold
for all polynomials at order $k-1$.

If $\overline k\in S$, put
$T=S\setminus\{\overline k\}$.
We can obtain a word counted by $A_{k,S}$ by inserting
the pair $kk$ into a word counted by $A_{k-1,T}$.
Hence $A_{k,S}
 =D_{3k-|S|-2}A_{k-1,T}$.
The polynomial on the right before insertion has degree
$2k-|S|-2$, so the parameter of $D$ exceeds its degree
by $k$. By induction and using
Lemma~\ref{lem:top-heavy-operator}, we find that $A_{k,S}(x)$ is strictly top heavy and has an increasing
left half.

If $\overline k\notin S$, we insert $\overline k$
and then the pair $kk$. Thus
$A_{k,S}
 =D_{3k-|S|-2}D_{3k-|S|-3}A_{k-1,S}$.
The polynomial $A_{k-1,S}$ has degree
$2k-|S|-3$. The parameter in the first insertion operator
exceeds this degree by $k$. Insertion raises the degree
by one, so the same is true for the second insertion operator.
Applying Lemma~\ref{lem:top-heavy-operator} twice proves
the strict inequalities for $A_{k,S}$.
Recall that $$A_{k,i}(x)
 =\sum_{\substack{S\subseteq[\overline k]\\|S|=i}}
 A_{k,S}(x).$$
Every summand has degree $2k-i-1$. Both coefficient
inequalities are preserved when polynomials of the same
degree are added. Hence they hold for $A_{k,i}(x)$. This completes the proof.
\end{proof}
%%%%%%%%%%%%%%%%%%%%%%%%%%%%%%%%%%%%%%%%%%%%%%%%%
%%%%%%%%%%%%%%%%%%%%%%%%%%%%%%%%%%%%%%%%%%%%%%%%%%%%%%%%
%%%%%%%%%%%%%%%%%%%%%%%%%%%%%%%%%%%%%%%%%%%%
%%%%%%%%%%%%%%%%%%%%%%%%%%%%%%%%%%%%%%%%%%%
%%%%%%%%%%%%%%%%%%%%%%%%%%%%%%%%%%%%%%%%%%%%

\end{document}